\documentclass{amsart}
\usepackage{amsmath,amscd,amssymb,stmaryrd}
\usepackage{amsfonts}
\usepackage[all]{xy}
\usepackage{mathtools}
\usepackage{enumerate}
\usepackage[retainorgcmds]{IEEEtrantools}
\usepackage{color}
\newif\ifshowchanges
\showchangestrue

\numberwithin{equation}{subsection}
\theoremstyle{plain}
\newtheorem{thm}[subsection]{Theorem}

\newtheorem{prop}[subsection]{Proposition}
\newtheorem{lemma}[subsection]{Lemma}

\theoremstyle{definition}
\newtheorem*{ackn}{Acknowledgement}
\newtheorem{defn}[subsection]{Definition}
\newtheorem{example}[subsection]{Example}

\newtheorem{cont}[subsection]{Contents}
\theoremstyle{remark}
\newtheorem{rem}[subsection]{Remark}
\newtheorem{rems}[subsection]{Remarks}

\begin{document}
\title{The minimal width of universal $p$-adic ReLU neural networks}
\author{S\'andor Z. Kiss, Ambrus P\'al}
\date{November 4, 2026.}
\address{Mathematical Institute, E\"otv\"os Lor\'and University
1117 Budapest, P\'azm\'any P\'eter s\'et\'any 1/C, Hungary;\newline \hspace*{4mm} 
Department of Mathematical Sciences, UAEU,  UAE}
\email{ambrus.pal@ttk.elte.hu}

\address{Department of Algebra and Geometry, Institute of Mathemetics, Budapest University of Technology and Economics, M\H{u}egyetem rkp. 3., H-1111 Budapest, Hungary; \newline \hspace*{4mm} 
HUN-REN Alfr\'ed R\'enyi Institute of Mathematics, Re\'altanoda utca 13--15., H-1053 Budapest,  Hungary; \newline \hspace*{4mm}}
\email{kiss.sandor@ttk.bme.hu}
\begin{abstract} We determine the minimal width of $p$-adic neural networks with the universal approximation property for continuous
$\mathbb Q_p$-valued functions on compact open subsets with respect to the $L_q$ norms and the $C_1$ norm, where the activation function is a natural $p$-adic analogue of the ReLU function.
\end{abstract}
\maketitle
\pagestyle{myheadings}
\markboth{S\'andor Z. Kiss, Ambrus P\'al}{The minimal width of universal $p$-adic ReLU neural networks}

\section{Introduction}

Many problems where neural networks are applied are inherently the approximation of a function whose values are either zero or one, such as classification of pictures with a cat or without a cat. The usual approach is to interpret the input as a vector in a finite dimensional vector space, to use a neural network with real weights and a real activation function, and to produce a real valued approximation of the target function. A particularly successful class of neural networks for these class of problems are deep networks with ReLU or something similar as the activation function, so consequently there are many papers about the universality of such networks, and in particular the minimal width of such universal deep networks (see \cite{HS}, \cite{KMP}, \cite{LPWHW} and
\cite{PYLS}, for example).

There is no apparent reason why one should not use the field $\mathbb Q_p$ of $p$-adic numbers instead, which has all the structure to have a well-defined problem. In fact we could argue that $p$-adic neural networks, where the role of the reals $\mathbb R$ is played by $\mathbb Q_p$, is more suitable for such classification problems as the one mentioned above. The primary motivation for this paper is to demonstrate this philosophy by completely answering a close $p$-adic analogue of a problem which has been intensively studied over the reals for several years. 

More precisely, we are interested in approximating $\mathbb Z_p$-valued, or more generally $\mathbb Q_p$-valued functions on $\mathbb Z_p^n$, not approximating $\mathbb C_p$-valued functions, where
$\mathbb C_p$ is the Tate field, playing a role analogous to $\mathbb C$ in $p$-adic analysis, nor real valued functions on $\mathbb Z_p^n$. Moreover we will look at the minimal minimal width problem when we fix s particularly simple activation function closely analogous to ReLU, not when we are allowed to use any activation function or at least a class of activation functions, unlike the papers \cite{Zu} or \cite{ZN}. (The latter also consider fundamentally different types of neural networks.) There is another additional subtle choice, namely we allow $\mathbb Q_p$-valued weights on the edges even while approximating $\mathbb Z_p$-valued functions, since with our choice of activation function universality for $\mathbb Z_p$-valued weights is not possible (see Remark \ref{rem1.3} below).

The topological group $\mathbb Z_p^n$ is compact, so it has a unique Haar measure, normalised so that $\mathbb Z_p^n$ has measure $1$, and hence we can define $L_q$-norms on
$\mathbb Q_p$-valued functions by taking the absolute value. (Note however it is not possible to define integration for $\mathbb Z_p$-valued continuous functions on $\mathbb Z_p^n$, see Remark \ref{rem2.8} below). Moreover we have an obvious analogue of the $C_1$ norm, which we will call the $L_{\infty}$ norm, therefore we can make sense of the analogue of all the usual approximation problems. Our choice of activation function is closely analogous to ReLu, and it is also very simple to compute.
\begin{defn} The pReLU activation function $\mathrm{pReLU}:\mathbb Q_p\to\mathbb Q_p$ is defined as
$$\mathrm{pReLU}(x)=\begin{cases}
    x & \text{, if } x\in\mathbb Z_p,\\
    0 & \text{, otherwise. } \end{cases}$$
\end{defn}
Our main result is:
\begin{thm}\label{main1.2} For every $q\in [1,\infty]$ the $\mathrm{pReLU}$-networks of width $w$ have the universal approximation property for continuous functions $f:\mathbb Z_p^{d_x} \to\mathbb Q_p^{d_y}$ in the $L_q$ norm if and only if $w\geq\max(d_x+1,d_y)$.
\end{thm}
\begin{rem}\label{rem1.3} The claim above is not true if we only allow weights in $\mathbb Z_p$. If all weights and biases lie in $\mathbb Z_p$, then taking pReLU has no effect on $\mathbb Z_p^{d_x}$. If only the weights (the linear parts of the affine maps) are required to lie in $\mathbb Z_p$ and the biases are arbitrary, then, as a function on $\mathbb Z_p^{d_x}$, every pre-activation is either integral valued, so that pReLU has no effect on it, or it has absolute value larger than $1$ everywhere because of a non-integral bias, so that pReLU maps it to zero identically. In both cases such networks only compute affine maps on $\mathbb Z_p^{d_x}$, and the latter do not have the universal approximation property.
\end{rem}
\begin{rem}\label{rem1.4} Let $K\subset\mathbb Q_p^{d_x}$ be a compact, open subset. Then it is bounded, and hence there is a $0\neq c\in\mathbb Q_p$ and a $d\in\mathbb Q_p^{d_x}$ such that $cK+d\subseteq
\mathbb Z_p^{d_x}$. Since $cK+d$ is both open and closed, the extension of every continuous function $f:cK+d\to\mathbb Q_p^{d_y}$ by zero is a continuous function on $\mathbb Z_p^{d_x}$. Therefore the upper bound in Theorem \ref{main1.2} holds for every compact, open subset $K$ of $\mathbb Q_p^{d_x}$, where the $L_q$ norms are defined using the Haar measure $\mu_K$ of $K$ normalised by $\mu_K(K)=1$. The lower bound holds as well if $K$ is nonempty, so the generalisation of Theorem \ref{main1.2} holds for every nonempty compact, open subset of $\mathbb Q_p^{d_x}$ (for $K=\emptyset$ there is nothing to approximate). Indeed, since $K$ is open, it contains a ball $B=a+p^r\mathbb Z_p^{d_x}$, and $\varphi(x)=a+p^rx$ is an affine bijection $\mathbb Z_p^{d_x}\to B$. If $w<\max(d_x+1,d_y)$, then the proof of the lower bound in Section 2 gives a continuous function $f_0:\mathbb Z_p^{d_x}\to\mathbb Q_p^{d_y}$ and a $\delta_0>0$ such that $\|f_0-g_0\|_1\geq\delta_0$ for every $\mathrm{pReLU}$-network $g_0$ of width $w$. Let $f:K\to\mathbb Q_p^{d_y}$ be $f_0\circ\varphi^{-1}$ on $B$ and zero on $K\setminus B$; it is continuous, since $B$ is both open and closed. For every $\mathrm{pReLU}$-network $g$ of width $w$ the function $g\circ\varphi$ is again such a network (the affine map $\varphi$ is absorbed by the first affine map of $g$), and $\varphi$ carries the normalised Haar measure of $\mathbb Z_p^{d_x}$ to the Haar measure of $B$ normalised to total mass one, which is $\mu_K(B)^{-1}\mu_K|_B$. Hence $$\|f-g\|_q\geq\mu_K(B)^{1/q}\|f_0-g\circ\varphi\|_q\geq\mu_K(B)^{1/q}\delta_0,$$ where $\mu_K(B)^{1/q}$ is to be read as $1$ for $q=\infty$. So for the uniform norm the lower bound is unchanged, and for finite $q$ it is multiplied by the fixed positive factor $\mu_K(B)^{1/q}$. (We used that $\|\cdot\|_q\geq\|\cdot\|_1$ for the normalised measure.)
\end{rem}
Note that unlike in the real case, there is no discrepancy between the lower and upper bound in the $C_1$ norm case, and the bound is the same as for all $L_q$ norms. The fundamental reason for this is that the topology of $\mathbb Q_p$, unlike the topology of $\mathbb R$, is totally disconnected, so the subtle topological obstructions in the works cited above do not arise. Moreover this fact leads to an obvious strategy for the upper bound as follows: every function in both of case can be approximated by functions which are locally constant, and the latter are always functions whose value only depend on the coordinates mod $p^k$ for a suitable $k$, using that $\mathbb Z^n_p$ is compact and totally disconnected. In fact as a stepping stone to our main result, we show the following
\begin{thm}\label{upper2} Let $w\geq d_x+1$ and let $f:\mathbb Z_p^{d_x}\to\mathbb Q_p$ be a locally constant function. Then there is a $\mathrm{pReLU}$-network of width $w$ which computes $f$.
\end{thm}
\begin{rems} This result implies that the characteristic function of any measurable subset of $\mathbb Z_p^{d_x}$ can be approximated arbitrarily well in the $L_q$ norm for every $1\leq q<\infty$, by a $\mathrm{pReLU}$-network of width $w$ which is also a characteristic function, so at least we do not need to apply a final decision function. Indeed, for $1\leq q<\infty$ every measurable set $A$ can be approximated by open and closed sets $U$, since $\|1_A-1_U\|_q=\mu(A\mathbin\triangle U)^{1/q}$ can be made arbitrarily small, and the indicator $1_U$ of a clopen set is locally constant, so it is computed exactly by a network of width $w$ by Theorem \ref{upper2}.

The corresponding claim is false for $q=\infty$, both for the supremum and for the essential supremum. For example, in dimension one let $A=\bigcup_{k\geq0}\bigl(p^{2k}\mathbb Z_p\setminus p^{2k+1}\mathbb Z_p\bigr)$. Every neighbourhood of $0$ contains subsets of $A$ and of its complement of positive measure. If a continuous function $g$ satisfied $\|1_A-g\|_{\infty}<1$, then by the ultrametric inequality we would have $|g(x)-g(y)|=1$ for all such $x\in A$ and $y\notin A$ (outside a null set in the essential supremum case, which does not matter), contradicting the continuity of $g$ at $0$. So the error is at least one for every network.  
\end{rems}
The proof of Theorem \ref{upper2} proceeds in two steps. First for every positive integer $m$ we produce an encoding function, that is, a
$\mathrm{pReLU}$-network of width $d_x+1$ which is constant on the cosets of $p^m\mathbb Z_p^{d_x}$ and maps different cosets to different values in $\mathbb Z_p$. This reduces the theorem to the a priori much easier problem of interpolating $\mathbb Q_p$-valued functions on finite subsets of $\mathbb Z_p$ by $\mathrm{pReLU}$-networks of width $2$.

We conclude the proof of the upper bound in Theorem \ref{main1.2} by constructing for every $m$ as above a decoding function from
$\mathbb Z_p$ into $\mathbb Z_p^{d_y}$, that is, a $\mathrm{pReLU}$-network of width $d_y$ whose image intersects each coset of $p^m\mathbb Z_p^{d_y}$ in $\mathbb Z_p$. The construction of the latter which is perhaps the most divergent from the constructions in \cite{HS} and \cite{PYLS} at the technical level, although as the readers will appreciate, the overall structure and trickery are very similar. 

The proof of a lower bound is always intriguing, since one needs to find an obstruction to the approximation. This could be a deep matter, since showing that something is not possible requires usually some invariant to measure the distance between possible solutions and the goal. The proof of the bound $w\geq d_y$ is easy, because we can use, just like in the real case, that otherwise the image of such neural networks lie in a proper affine subspace. The proof of the bound $w\geq d_x+1$ is more subtle, the key fact being the following result, which has no precise real analogue:
\begin{thm} Let $f$ be a $\mathrm{pReLU}$-network of input dimension and width $n$. Then either $f|_{\mathbb Z_p^n}$ is an affine map or there is a ball $B\subset\mathbb Z_p^n$ of radius $\frac{1}{p}$ such that $f$ is constant in some direction on $B$. 
\end{thm}
Here constancy in some direction means that there is an $h\in\mathbb Q_p^n$ of norm $\frac{1}{p}$ such that $f(x+h)=f(x)$ for every $x\in B$. This claim was inspired by the proof Lemma 6 of \cite{HS}, although our actual argument is rather algebraic. However we tried to phrase it in a language which emphasises the analogy. We apply it to homeomorphisms $\mathbb Z_p^{d_x}\to\mathbb Z_p$ to derive a contradiction. The existence of the latter might be surprising to those who are unfamiliar $p$-adic analysis, but it is a well-known fact.  
\begin{cont} In the next section we give a quick overview of the basics of $p$-dic numbers. We prove the lower bound in the third section, and we prove the upper bound in the final, fourth section. 
\end{cont}
\begin{ackn} 
S\'andor Z. Kiss was supported by the NKFIH Grants No. K146387, KKP 144059 and the National Research, Development and Innovation Office NKFIH (Excellence program, Grant Nr. 153829). After the first version of this paper was written we used Chat GPT and Claude for revision. Most of the errors found were incorrect indices. The authors reviewed the suggested corrections, and take full responsibility for the final version of the paper. 

\end{ackn}

\section{A primer on $p$-adic numbers}

In this section we give a short summary about the basics of the $p$-adic numbers. Interested reader can find more details in the books \cite{Goueva} and \cite{Koblitz}. Let $p$ be a prime number. Every nonzero rational number $x$ can be written in the form
\[
x = p^{k}\frac{a}{b},
\]
where $a,b \in \mathbb{Z}$ are not divisible by $p$ and $k \in \mathbb{Z}$.
The $p$-adic absolute value of $x$ is defined by
\[
|x|_{p} = p^{-k},
\]
and
\[
|0|_{p} = 0.
\]

This defines a non-Archimedean absolute value on $\mathbb{Q}$. It satisfies the ultrametric inequality
\[
|x+y|_{p} \leq \max(|x|_{p},|y|_{p}).
\]

The field $\mathbb{Q}_p$ of $p$-adic numbers is defined as the completion of the field of rational numbers with respect to the $p$-adic absolute value.
This construction is analogous to the construction of the real numbers $\mathbb{R}$ as the completion of $\mathbb{Q}$ with respect to the usual absolute value.

The ring of $p$-adic integers, denoted by $\mathbb{Z}_p$, is defined as
\[
\mathbb{Z}_p = \{ x \in \mathbb{Q}_p : |x|_p \le 1 \}.
\]
It forms a commutative subring with identity. This ring is compact with respect to the ultrametric valuation, and open in $\mathbb Q_p$. Therefore $\mathbb Q_p$ is locally compact. Moreover $\mathbb{Z}_p$ is a principal ideal domain, where the unique maximal ideal is generated by $p$. Hence every non-zero ideal of $\mathbb{Z}_p$ is of the form $(p^n)$ for a unique positive integer $n$. Also every element of $\mathbb Z_p$ can be approximated by rational integers modulo powers of $p$.

 Every $x \in \mathbb{Z}_p$ can be uniquely expressed as
\[
x = a_0 + a_1 p + a_2 p^2 + a_3 p^3 + \dots,
\]
where $a_i \in \{0,1,\dots,p-1\}$.  
This provides an infinite $p$-adic expansion, analogous to decimal expansions, but with powers of $p$ increasing to the left.
Intuitively, each $p$-adic integer corresponds to a consistent sequence
\[
(x_1 \bmod p, x_2 \bmod p^2, x_3 \bmod p^3, \dots)
\]
with $x_{n+1} \equiv x_n \pmod{p^n}$ for all $n \ge 1$.

\section{Lower bound} 
 
We start with an overview of some geometry in vector spaces over $\mathbb Q_p$.
\begin{defn} Let $V$ be a finite dimensional $\mathbb Q_p$-linear vector space. We say that a subset $C\subset V$ is convex if it is empty or a coset of a $\mathbb Z_p$-submodule of $V$. Note that $V$ has a canonical topology which is induced by any norm on $V$, since all of these are equivalent, by the local compactness of $\mathbb Q_p$. When we talk about continuity, we mean this topology. With respect to this topology every linear map is continuous. 
\end{defn}
\begin{example} Let $|\cdot|$ denote the norm on $\mathbb Q_p$ normalised so that $|p|=\frac{1}{p}$. Then $\|\cdot\|$ on $\mathbb Q_p^n$ given by the rule
$$\|(x_1,\ldots,x_n)\|=\max(|x_1|,\ldots,|x_n|)$$
is a norm. For every $v\in\mathbb Q_p^n$ and real $\rho>0$ let
$$B(v,\rho)=\{u\in\mathbb Q_p^n\mid \|u-v\|\leq\rho\}
\quad\textrm{and}\quad
S(v,\rho)=\{u\in\mathbb Q_p^n\mid \|u-v\|=\rho\}$$
be the ball and sphere of centre $v$ and radius $\rho$, respectively. Then $B(v,\rho)$ is both compact and convex, while $S(v,\rho)$ is compact. (The sphere is in general not convex, for example $S(0,1)=\mathbb Z_p^n\setminus p\mathbb Z_p^n$ is not convex for $n\geq2$, although $S(0,1)=1+2\mathbb Z_2$ is convex for $p=2$, $n=1$. Only the compactness of spheres is used below.) In particular $\mathbb Z_p^n=B(0,1)\subset\mathbb Q_p^n$ is compact and convex.
\end{example}
The following claims are well-known, but we include it for the reader's convenience. 
\begin{lemma}\label{convex} Let $V$ be as above. The following holds:
\begin{enumerate}
\item[$(i)$] if $f:V\to W$ is an affine map between finite dimensional
$\mathbb Q_p$-linear vector spaces, and $C\subset W$ is convex, then $f^{-1}(C)$ is also convex,
\item[$(ii)$] the intersection of convex subsets of $V$ are convex.
\end{enumerate}
\end{lemma}
\begin{proof} Claim $(i)$ is trivial when $C$ is empty. Otherwise $C$ is of the form $L+u$, where $L$ is a $\mathbb Z_p$-submodule of $W$ and $u$ is any element of $C$. Note that for every $w\in W$ the translate $C+w=L+u+w$ is also convex by definition, so we may assume without the loss of generality that $f$ is $\mathbb Q_p$-linear. Then $f^{-1}(L)$ is a $\mathbb Z_p$-submodule of $V$. If the image of $f$ does not intersect  $C$ then $f^{-1}(C)$ is empty, so $(i)$ holds. Otherwise there is a $v\in V$ such that $f(v)\in C$ and hence $C=L+f(v)$. Since $f$ is additive, we get that $f^{-1}(C)=f^{-1}(L+f(v))=f^{-1}(L)+v$, so it is is convex. 

Now let $\{C_i\}_{i\in I}$ be a set of convex subsets of $V$. If their intersection is empty, claim $(ii)$ holds, so we may assume that there is an $h\in\bigcap_{i\in I}C_i$ without the loss of generality. Then $h\in C_i$ for each $i\in I$, so for every such index there is a $\mathbb Z_p$-submodule $L_i$ of $V$ such that $C_i=L_i+h$. Then
$$\bigcap_{i\in I}C_i=\bigcap_{i\in I}(L_i+h)=\bigcap_{i\in I}L_i+h.$$
Since $\bigcap_{i\in I}L_i$ is the intersection of $\mathbb Z_p$-submodules of $V$, it is a $\mathbb Z_p$-submodule of $V$, and hence 
$\bigcap_{i\in I}C_i$ is convex.
\end{proof}
\begin{prop}\label{pidgeon} Let $C$ be a proper convex subset of $\mathbb Z_p^n$. Then there is a ball of radius $\frac{1}{p}$ in $\mathbb Z_p^n$ disjoint from $C$. 
\end{prop}
\begin{proof} The claim is trivial when $C$ is empty. If $C$ is not empty, then it is of the form $L+u$, where $L$ is a $\mathbb Z_p$-submodule of $\mathbb Q_p^n$ and $u$ is any element of $C$. In this case $L=C-u$. As $\mathbb Z_p^n$ is a subgroup of $\mathbb Q_p^n$ we get that $L\subset\mathbb Z_p^n$. If $B$ is a ball of radius $\frac{1}{p}$ in
$\mathbb Z_p^n$ disjoint from $L$, then $B+u$ is a ball of radius $\frac{1}{p}$ in $\mathbb Z_p^n$ disjoint from $C$, so we may assume that $C$ is a proper $\mathbb Z_p$-submodule of $\mathbb Z_p^n$ without the loss of generality. 

Since $\mathbb Z_p$ is Noetherian and $\mathbb Z_p^n$ is finitely generated, there is a proper maximal $\mathbb Z_p$-submodule $C'
\subset\mathbb Z_p^n$ containing $C$. We may assume without the loss of generality that $C=C'$. Then the quotient $\mathbb Z_p^n/C$ is a simple $\mathbb Z_p$-module. It is also finitely generated, since it is the quotient of a finitely generated $\mathbb Z_p$-module. Since $\mathbb Z_p$ is a principal ideal domain, by the structure theorem for finitely generated modules over a principal ideal domain, the quotient $\mathbb Z_p^n/C$ is a finite direct sum of cyclic $\mathbb Z_p$-modules.

Since $\mathbb Z_p^n/C$ is also simple, it must be cyclic. A simple cyclic $\mathbb Z_p$-module is of the form $\mathbb Z_p/\mathfrak m$, where $\mathfrak m\triangleleft\mathbb Z_p$ is a maximal ideal. But $\mathbb Z_p$ is a local ring, so its unique maximal ideal is $(p)$, and hence
$\mathbb Z_p^n/C$ is $\mathbb Z_p/p\mathbb Z_p$. Therefore $C$ contains $p\mathbb Z_p^n$. Let $x\in\mathbb Z_p^n-C$; since $C$ is a $\mathbb Z_p$-submodule of $\mathbb Z_p^n$ the coset $C+x$ is disjoint from $C$. By the above the ball $p\mathbb Z_p^n+x$ of radius $\frac{1}{p}$ is a subset of $C+x$, so it is also disjoint from $C$.
\end{proof}
\begin{thm}\label{brouwer} There is a homeomorphism $\mathbb Z_p^n\to\mathbb Z_p^m$ for every pair of positive integers $m,n$.  
\end{thm}
\begin{proof} This claim is a special case of what is known as Brouwer's theorem (see Theorem 30.3 of \cite{Wi} on page 216): let $X,Y$ be metrisable totally disconnected compact topological spaces without isolated points. Then $X$ and $Y$ are homeomorphic. 
\end{proof}
\begin{prop}\label{uniform1} Let $f:\mathbb Z_p^n\to\mathbb Z_p$ be an injective continuous function. Then there is an $\epsilon>0$ such that $|f(x+h)-f(x)|\geq\epsilon$ for every $x\in\mathbb Z_p^n$ and every $h\in S(0,\frac{1}{p})$. 
\end{prop}
\begin{proof} Note that the function $g:\mathbb Z_p^n\times
S(0,\frac{1}{p})\to\mathbb R$ given by the rule
$$(x,h)\mapsto|f(x+h)-f(x)|$$
is continuous, and only takes positive values, since $f$ is injective. Therefore by the compactness of its domain $\mathbb Z_p^n\times
S(0,\frac{1}{p})$ the function $g$ takes a positive minimal value. 
\end{proof}
\begin{defn} The topological group $\mathbb Z_p^n$ is compact, so it has a unique Haar measure $\mu$ with $\mu(\mathbb Z_p^n)=1$, which we call the normalised Haar measure. For every $q\geq1$ and continuous function $f:\mathbb Z_p^n\to\mathbb Q_p^m$ we define the $L_q$-norm of $f$ as
$$\|f\|_q=\left(\int_{\mathbb Z_p^n}\|f(x)\|^qd\mu(x)\right)^{\frac{1}{q}}.$$
We also define the $L_{\infty}$-norm of $f$ as
$$\|f\|_{\infty}=\sup_{x\in \mathbb Z_p^n}\|f(x)\|.$$
Since $\mathbb Z_p^n$ is compact and $\|f(x)\|$ is continuous, the $L_{\infty}$-norm is finite. Moreover for every $q_1,q_2\in[1,\infty]$ with $q_1\leq q_2$ we have 
$$\|f\|_{q_1}\leq\|f\|_{q_2}$$
as a consequence of H\"older's inequality and the normalisation $\mu(\mathbb Z_p^n)=1$, so the $L_q$-norm is finite for every $q\in[1,\infty]$. 
\end{defn}
\begin{rem}\label{rem2.8} Although we can define $L_q$-norms on $\mathbb Q_p$-valued continuous functions on $\mathbb Z_p^n$ it is not possible to define a $\mathbb Q_p$-valued integration for such functions. The precise claim is the following: let $I:C(\mathbb Z_p^n,\mathbb Q_p)
\to\mathbb Q_p$ be a $\mathbb Q_p$-linear map which is continuous with respect to the $L_{\infty}$-norm and invariant with respect to translations. Then $I$ is identically zero. 
\end{rem}
\begin{defn} Let $B$ be a ball of radius $\rho>0$. We say that a function $g:B\to\mathbb Q_p$ is constant in some direction on $B$ if there is an 
$h\in S(0,\rho)$ such that $g(x+h)=g(x)$ for every $x\in B$. Note that by the ultrametric inequality we have $x+h\in B$ if $x\in B$ and $h\in S(0,\rho)$, so our condition is well-posed. 
\end{defn}
\begin{prop}\label{uniform2} Let $f:\mathbb Z_p^n\to\mathbb Z_p$ be an injective continuous function. Then there is an $\delta>0$ such that for every continuous function $g:\mathbb Z_p^n\to\mathbb Q_p$ which is constant in some direction on a ball of radius $\frac{1}{p}$ in $\mathbb Z_p^n$ we have $\|f(x)-g(x)\|_1\geq\delta$. 
\end{prop}
\begin{proof} By Proposition \ref{uniform1} there is an $\epsilon>0$ such that $|f(x+h)-f(x)|\geq\epsilon$ for every $x\in\mathbb Z_p^n$ and every $h\in S(0,\frac{1}{p})$. Let $g:\mathbb Z_p^n\to\mathbb Q_p$ be a continuous function such that for the ball $B\subset\mathbb Z_p^n$ of radius $\frac{1}{p}$ and $h\in S(0,\frac{1}{p})$ we have $g(x+h)=g(x)$ for every $x\in B$. We will show that $\|f(x)-g(x)\|_1\geq\frac{\epsilon}{2p^n}$. Assume the contrary; then 
$$\frac{\epsilon}{2p^n}>\|f(x)-g(x)\|_1=\int_{\mathbb Z_p^n}\|f(x)-g(x)\|d\mu(x)
\geq\int_{B}\|f(x)-g(x)\|d\mu(x),$$
and
\begin{align*}
\frac{\epsilon}{p^n}&>\int_{B}\|f(x)-g(x)\|d\mu(x)+\int_{B}\|f(x+h)-g(x+h)\|d\mu(x)  \\
& \geq\int_{B}\|f(x)-f(x+h)+g(x+h)-g(x)\|d\mu(x) \\
& =\int_{B}\|f(x)-f(x+h)\|d\mu(x)\geq\epsilon\cdot\mu(B),
\end{align*}
using the translation-invariance of the Haar measure in the first line, the triangle inequality for the norm $\|\cdot\|$ in the second inequality and using that $g(x+h)=g(x)$ for every $x\in B$ in the first equation in the third line. Since $\mathbb Z_p^n$ is the disjoint union of $p^n$ cosets of $p\mathbb Z_p^n$, and since the Haar measure is translation-invariant, we get that $\mu(B)=\mu(p\mathbb Z_p^n)=\frac{1}{p^n}$, which is a contradiction. 
\end{proof}
\begin{defn} By an activation function we mean a continuous function $f:\mathbb Q_p\to\mathbb Q_p$. The main example of activation function we will concern ourselves with is pReLU which is defined as
$$\mathrm{pReLU}(x)=\begin{cases}
    x & \text{, if } x\in\mathbb Z_p,\\
    0 & \text{, otherwise. } \end{cases}$$
\end{defn}
\begin{defn} Given a set of activation functions $\Sigma$, an $L$-layer $\Sigma$-neural network $f$ of input dimension $d_x$, output dimension $d_y$, and hidden layer dimensions $d_1,\ldots,d_{L-1}$ is composition
$$f=t_L\circ\Sigma_{L-1}\circ\cdots\circ t_2 \circ\Sigma_1 \circ t_1
:\mathbb Q_p^{d_x}\to\mathbb Q_p^{d_y},$$
where $t_l:\mathbb Q_p^{d_{l-1}}\to\mathbb Q_p^{d_l}$ is an affine map and $\Sigma_l$ is a vector of activation functions:
$$\Sigma_l(x_1,\ldots,x_{d_l}) =(\rho^l_1(x_1),\ldots,\rho^l_{d_l}(x_{d_l})),$$
where $\rho^l_i\in\Sigma$. When $\Sigma$ is a singleton, that is, 
$\Sigma=\{\rho\}$, we will call $\Sigma$-neural networks simply 
$\rho$-networks. We define the width $w$ of $f$ as the maximum over $d_1,\ldots,d_{L-1}$. (Here we assume that $d_x,d_y\geq1$ and $L\geq2$, that is, there is at least one hidden layer, so that this maximum is taken over a non-empty list.)
\end{defn}
The following result is the key step in our proof:
\begin{thm}\label{key} Let $f$ be a $\mathrm{pReLU}$-network of input dimension and width $n$. Then either $f|_{\mathbb Z_p^n}$ is an affine map or there is a ball $B\subset\mathbb Z_p^n$ of radius $\frac{1}{p}$ such that $f$ is constant in some direction on $B$. 
\end{thm}
\begin{proof} Assume that $f$ is of the form
$$f=t_L\circ\Sigma\circ\cdots\circ t_2 \circ\Sigma\circ t_1
:\mathbb Q_p^{d_x}\to\mathbb Q_p^{d_y},$$
where 
$$\Sigma(x_1,\ldots,x_n) =(\mathrm{pReLU}(x_1),\ldots,\mathrm{pReLU}(x_n)).$$
By extending the affine maps $t_1,t_2,\ldots,t_{L-1}$ by zero, if necessary, we may assume without the loss of generality that all hidden layer dimensions are $n$. Set
$$f_k=t_k\circ\Sigma\circ\cdots\circ t_2 \circ\Sigma\circ t_1$$
for $k=1,2,\ldots,L-1$. First assume that for each such $k$ we have $f_k(\mathbb Z_p^n)\subset\mathbb Z_p^n$. Since $\Sigma|_{\mathbb Z_p^n}$ is the indentity, we get that
$$f|_{\mathbb Z_p^n}=t_L\circ\cdots\circ t_2\circ t_1|_{\mathbb Z_p^n}$$
is a composition of affine maps, so it is affine. Otherwise there is a smallest $k$ such that $f_k^{-1}(\mathbb Z_p^n)\cap\mathbb Z_p^n\subsetneq\mathbb Z_p^n$. For every $i=1,2,\ldots,n$ let $S_i$ denote the set
$$S_i=\{(x_1,\ldots,x_n)\in\mathbb Q_p^n\mid x_i\in\mathbb Z_p\};$$
this set is a $\mathbb Z_p$-submodule, so it is convex. Clearly $\mathbb Z_p^n=S_1\cap\cdots\cap S_n$, and hence
$$f_k^{-1}(\mathbb Z_p^n)=f_k^{-1}(S_1)\cap\cdots\cap f_k^{-1}(S_n).$$
Therefore there is an $i$ such that $f_k^{-1}(S_i)\cap\mathbb Z_p^n\subsetneq\mathbb Z_p^n$. Since $k$ is minimal, we get that
$$f_k|_{\mathbb Z_p^n}=
t_k\circ\cdots\circ t_2\circ t_1|_{\mathbb Z_p^n}.$$
Since $t_k\circ\cdots\circ t_2\circ t_1$ is a composition of affine maps, it is affine, so $(t_k\circ\cdots\circ t_2\circ t_1)^{-1}(S_i)$ is convex by 
part $(i)$ of Lemma \ref{convex}, and hence
$$f_k^{-1}(S_i)\cap\mathbb Z_p^n=(t_k\circ\cdots\circ t_2\circ t_1)^{-1}(S_i)\cap\mathbb Z_p^n$$
is an intersection of convex subsets, so it is convex by part $(ii)$ of Lemma \ref{convex}. Because it is a proper subset of $\mathbb Z_p^n$ we get that there is a ball $B$ of radius $\frac{1}{p}$ in $\mathbb Z_p^n$ disjoint from it by Proposition \ref{pidgeon}.

We will show that $f$ is constant in some direction on $B$. Since $f$ is the composition of $\Sigma\circ f_k$ with some function, it will be sufficient to show the same for $\Sigma\circ f_k$. Let $l:\mathbb Q_p^n
\to\mathbb Q_p^n$ be a linear map and let $v\in\mathbb Q_p^n$ be a vector such that $f_k(x)=l(x)+v$ for every $x\in\mathbb Z_p^n$. Let $L_i\subset\mathbb Q_p^n$ be the one-dimensional linear subspace
$$L_i=\{(x_1,\ldots,x_n)\in\mathbb Q_p^n\mid x_j=0\ (\forall j\neq i)\}.$$
Let $\pi:\mathbb Q_p^n\to\mathbb Q_p^n/L_i$ be the quotient map. Then $l^{-1}(L_i)=\ker(\pi\circ l)$, and $\pi\circ l$ is a linear map from an $n$-dimensional space to an $(n-1)$-dimensional one, so $l^{-1}(L_i)$ has dimension at least one. Let $h$ be a non-zero vector in this subspace: by rescaling, if it is necessary, we may assume that $\|h\|=\frac{1}{p}$. By the above $f_k(B)$ is disjoint from $S_i$, so if $x,y\in B$ are such that $f_k(x)$ and $f_k(y)$ have the same $j$-th coordinate for every $j\neq i$, or equivalently $f_k(x)-f_k(y)\in L_i$, then
$\Sigma(f_k(x))=\Sigma(f_k(y))$. In particular
$\Sigma(f_k(x+h))=\Sigma(f_k(x)+l(h))=\Sigma(f_k(x))$ for every $x\in B$. 
\end{proof}
\begin{defn} We say that  $\Sigma$-networks of width $w$ have the universal approximation property for continuous functions $f:\mathbb Z_p^{d_x}\to\mathbb Q_p^{d_y}$ in the $L_q$ norm, where $q\in[1,\infty]$, if for every such $f$ and $\epsilon> 0$ there exists a $\Sigma$-network $g$ of width $w$ such that $\|f-g\|_q<\epsilon$.
\end{defn}
\begin{thm} For every $q\in [1,\infty]$ the $\mathrm{pReLU}$-networks of width $w$ have the universal approximation property for continuous functions $f:\mathbb Z_p^{d_x} \to\mathbb Q_p^{d_y}$ in the $L_q$ norm only if $w\geq\max(d_x+1,d_y)$.
\end{thm}
\begin{proof} We need to show that if $w<\max(d_x+1,d_y)$ then there is a continuous function $f:\mathbb Z_p^{d_x} \to\mathbb Q_p^{d_y}$ which cannot be approximated arbitrarily well with $\mathrm{pReLU}$-networks of width $w$ in the $L_q$ norm for every $q\in [1,\infty]$. Since the $L_1$-norm is the weakest, it will be enough to show the claim for $q=1$ only. Let's suppose that the claim is false, and derive a contradiction. 

First assume that $w<d_y$. Note that the set $A$ of cosets of $p^{d_y}\mathbb Z_p$ in $\mathbb Z_p$  and the set
$$B=\{(x_1,\ldots,x_{d_y})\in\mathbb Z_p^{d_y}\mid x_i\in\{0,1,\ldots,p-1\}\ (\forall i=1,\ldots,d_y)\}$$
 have the same cardinality $p^{d_y}$, so there is a bijection $b:A\to B$. Let $h:\mathbb Z_p\to\mathbb Z_p^{d_y}$ be the function given by the rule $x\mapsto
b(x\ \mathrm{mod}\ p^{d_y}\mathbb Z_p)$. This function is continuous since it is constant on compact open subsets. Let $f:\mathbb Z_p^{d_x}\to\mathbb Z_p^{d_y}$ be the composition of the projection $\mathbb Z_p^{d_x}\to
\mathbb Z_p$ onto the first coordinate and $h$. Since $f$ is the composition of two continuous functions, it is also continuous. Its level sets are translates of each other, so they have the same Haar measure. Since they give a partition of $\mathbb Z_p^{d_x}$ into $p^{d_y}$ subsets, their measure is $\frac{1}{p^{d_y}}$. 

Now let $g:\mathbb Q_p^{d_x}\to\mathbb Q_p^{d_y}$ be a 
$\mathrm{pReLU}$-network of width $w$. Since $w<d_y$ the image of $g$ lies in a proper affine subspace of $\mathbb Q_p^{d_y}$. The intersection $C$ of the latter with $\mathbb Z_p^{d_y}$ is convex by part $(ii)$ of Lemma \ref{convex}, since it is the intersection of convex subsets, and it is a proper subset of $\mathbb Z_p^{d_y}$, since the latter contains the zero vector and spans
$\mathbb Q_p^{d_y}$. Therefore by Proposition \ref{pidgeon} there is a ball of radius $\frac{1}{p}$ in $\mathbb Z_p^{d_y}$ disjoint from $C$. Because $B$ is a full set of representatives of the cosets of $p\mathbb Z_p^{d_y}$ in $\mathbb Z_p^{d_y}$, we get that there is a $v\in B$ such that $v+p\mathbb Z_p^{d_y}$ is disjoint from $C$. Therefore
\begin{align*}
\|f(x)-g(x)\|_1&=\int_{\mathbb Z_p^{d_x}}\|f(x)-g(x)\|d\mu(x) \\
& \geq\int_{f^{-1}(v)}\|v-g(x)\|d\mu(x) \\
& \geq\int_{f^{-1}(v)}\frac{1}{p}d\mu(x)=
\frac{1}{p}\cdot\mu(f^{-1}(v))=\frac{1}{p^{d_y+1}},
\end{align*}
which is a contradiction. 

Otherwise $w\leq d_x$; write $n=d_x$. We first choose a continuous injective function $f_0:\mathbb Z_p^n\to\mathbb Z_p$ as follows. If $n\geq2$, let $f_0$ be a homeomorphism $\mathbb Z_p^n\to\mathbb Z_p$, which exists by Theorem \ref{brouwer}. If $n=1$, let $f_0(t)=t+pt^2$. This function is indeed injective, since for every $s,t\in\mathbb Z_p$ we have
$$f_0(s)-f_0(t)=(s-t)\bigl(1+p(s+t)\bigr),$$
and $1+p(s+t)$ is a unit of $\mathbb Z_p$, so $|f_0(s)-f_0(t)|=|s-t|$. Moreover $f_0$ is clearly continuous, and the inequality of Proposition \ref{uniform1} holds with $\epsilon=\frac1p$. 

Let $f:\mathbb Z_p^{d_x}\to\mathbb Z_p^{d_y}$ be the composition of $f_0$ and the embedding of $\mathbb Z_p$ into $\mathbb Z_p^{d_y}$ via the first coordinate, and let $\pi_1:\mathbb Q_p^{d_y}\to\mathbb Q_p$ be the projection onto the first coordinate. For every $\mathrm{pReLU}$-network $g$ of width $w$, replacing the last affine map $t_L$ of $g$ by $\pi_1\circ t_L$ we get a scalar $\mathrm{pReLU}$-network $g^*=\pi_1\circ g$ of the same width. Since $\pi_1$ is a contraction and $\pi_1\circ f=f_0$, we get that
$$\|f_0-g^*\|_1=\|\pi_1\circ f-\pi_1\circ g\|_1\leq\|f-g\|_1.$$
Therefore it will be sufficient to find a $\delta_0>0$ such that $\|f_0-g\|_1\geq\delta_0$ for every scalar $\mathrm{pReLU}$-network $g:\mathbb Q_p^n\to\mathbb Q_p$ of width $w\leq n$, because then $f$ cannot be approximated arbitrarily well, which is a contradiction. Let $g$ be such a network. Adding hidden coordinates which are identically zero we may assume that all hidden layers of $g$ have dimension $n$, so by Theorem \ref{key} either $g|_{\mathbb Z_p^n}$ is affine, or $g$ is constant in some direction on a ball of radius $\frac1p$ in $\mathbb Z_p^n$.

$(a)$ Assume that $g$ is constant in some direction on a ball of radius $\frac1p$. Then $\|f_0-g\|_1\geq\delta_1$ by Proposition \ref{uniform2}, where $\delta_1>0$ depends only on $f_0$. For $n=1$ the proof of Proposition \ref{uniform2} with $\epsilon=\frac1p$ gives $\delta_1=\frac{1}{2p^2}$.

$(b)$ Assume that $n\geq2$ and $g|_{\mathbb Z_p^n}$ is affine, that is, $g(x)=\lambda(x)+c$ on $\mathbb Z_p^n$ for a linear form $\lambda:\mathbb Q_p^n\to\mathbb Q_p$ and a constant $c\in\mathbb Q_p$. Since $n\geq2$, the kernel of $\lambda$ contains a non-zero vector, which we may rescale to have norm $\frac1p$. Then $g(x+h)=g(x)$ for every $x\in\mathbb Z_p^n$, so $g$ is constant in some direction on a ball of radius $\frac1p$, and we are in case $(a)$.

$(c)$ Finally assume that $n=1$ and $g(t)=at+b$ for every $t\in\mathbb Z_p$. Let $e=f_0-g$ on $\mathbb Z_p$. The affine function $g$ cancels in second differences, so for every $t\in\mathbb Z_p$ we have:
$$e(t+2)-2e(t+1)+e(t)=p\bigl((t+2)^2-2(t+1)^2+t^2\bigr)=2p.$$
By the triangle inequality we get that $|2p|\leq|e(t+2)|+|2|\,|e(t+1)|+|e(t)|$ for every $t\in\mathbb Z_p$. Integrating over $\mathbb Z_p$ and using that $\mu$ is translation invariant and $\mu(\mathbb Z_p)=1$, we get that $|2p|\leq(2+|2|)\|e\|_1$, so
$$\|f_0-g\|_1\geq\frac{|2p|}{2+|2|}>0.$$
Altogether $\|f_0-g\|_1\geq\delta_0$ for every scalar $\mathrm{pReLU}$-network of width $w\leq n$, where $\delta_0=\delta_1$ if $n\geq2$, and $\delta_0=\min\bigl\{\frac{1}{2p^2},\frac{|2p|}{2+|2|}\bigr\}$ if $n=1$. This is the desired contradiction. 
\end{proof}

\section{Upper bound} 

In this section we will give the upper bound for the $p$-adic neural network approximation problem.
\begin{defn} Let $X,Y$ be a topological spaces. A function $f:X\to Y$ is locally constant, if every $x\in X$ has an open neighbourhood $U_x\subseteq X$ such that $f|_{U_x}$ is constant. Clearly every locally constant function is continuous. When $X$ is connected then every locally constant function is actually constant. This is the case for example when $X$ is a convex in a real vector space, so in the real case we don't have many such functions usually. However when $X$ is totally disconnected there are an abundance of locally constant functions on $X$. This is the case when $X$ is a compact subset of a finite dimensional $\mathbb Q_p$-linear vector space.  
\end{defn}
The following claim is very well-known, We only include it for the reader's convenience. 
\begin{lemma}\label{locally_constant} Let $V\cong\mathbb Q^d_p$ be a finite dimensional $\mathbb Q_p$-linear vector space and let $C\subset V$ be compact. The following holds:
\begin{enumerate}
\item[$(i)$] the space of locally constant functions $f:C\to\mathbb Q^n_p$ is dense in the space of continuous functions with respect to the supremum norm, 
\item[$(ii)$] for every  locally constant function $f:C\to\mathbb Q^n_p$ there is a positive integer $m$ such that $f$ is constant on the intersection of $C$ and the cosets of $p^m\mathbb Z_p^d$. 
\end{enumerate}
\end{lemma}
\begin{proof} Let $f:C\to\mathbb Q^n_p$ be a continuous function and fix a positive $\epsilon$. Since $C$ is compact, the function $f$ is uniformly continuous, there is a positive $\delta$ such that for every $x,y\in C$ such that $\|x-y\|\leq\delta$ we have $\|f(x)-f(y)\|<\epsilon$. Note that $C$ can be covered by the cosets of the compact, open subgroup $B(0,\delta)\subset V$.

For every such coset $B$ such that $B\cap C\neq\emptyset$ choose an $x_B\in B\cap C$. Let $g:C\to\mathbb Q^n_p$ be the unique function such that for every $B$ as above $g|_{B\cap C}$ is the constant $f(x_B)$. Clearly $g$ is locally constant. Now let $y\in C$ be arbitrary; let $B$ be the unique coset of $B(0,\delta)$ such that $y\in B$. Then $\|f(y)-g(y)\|=\|f(y)-f(x_B)\|<\epsilon$, since $\|y-x_B\|\leq\delta$. Therefore claim $(i)$ holds. 

Let $f:C\to\mathbb Q^n_p$ be a locally constant function. For every $x\in C$ choose an open ball $B(x,\rho_x)$ such that $f|_{B(x,\rho_x)\cap C}$ is constant. Since $C$ is compact, there is a finite subset $x_1,x_2,\ldots,x_n\in C$ such that $B(x_1,\rho_{x_1}),\ldots,B(x_n,\rho_{x_n})$ already cover $C$. Set $\rho$ as the minimum of the radii
$\rho_1,\ldots,\rho_n$. Then each $B(x_i,\rho_{x_i})$ is the union of cosets   of the $\mathbb Z_p$-module $B(0,\rho)$, and hence we get that $C$ can  be covered by some cosets of $B(0,\rho)$ such that the restriction of $f$ onto the intersection of $C$ and any of these cosets is constant. Choose a positive integer $m$ such that $p^{-m}\leq\rho$. Then $p^m\mathbb Z_p^d=B(0,p^{-m})\subseteq B(0,\rho)$, so every coset of $p^m\mathbb Z_p^d$ is contained in a coset of $B(0,\rho)$, and claim $(ii)$ follows.
\end{proof}
Our aim is to give a proof of the following
\begin{thm}\label{upper} If $w\geq\max(d_x+1,d_y)$ then $\mathrm{pReLU}$-networks of width $w$ have the universal approximation property for continuous functions $f:\mathbb Z_p^{d_x} \to\mathbb Q_p^{d_y}$ in the $L_{\infty}$-norm.
\end{thm}
The first main step to prove the result above is the following
\begin{thm}\label{upper2b} Let $w\geq d_x+1$ and let $f:\mathbb Z_p^{d_x}\to\mathbb Q_p$ be a locally constant function. Then there is a $\mathrm{pReLU}$-network of width $w$ which computes $f$.
\end{thm}
\begin{proof} We will prove these theorems through a sequence of lemmas. 
\begin{lemma}\label{lemma4} Let $S\subset\mathbb Z_p$ be a finite subset, and let $\alpha,\beta\in\mathbb Z_p$ be two elements. Assume that $\alpha\not\in S$. Then there is a $\mathrm{pReLU}$-network of width $2$ which computes a function $f:\mathbb Z_p\to\mathbb Z_p$ such that $f(\alpha)=\beta$ and  $f(x)=x$, if $x\in S$. 
\end{lemma}
\begin{proof} If $\alpha=\beta$, then the identity function of $\mathbb Z_p$, which is computed by a $\mathrm{pReLU}$-network of width $1$, will do. So we may assume without the loss of generality that $\alpha\neq\beta$, which implies $|\beta-\alpha|>0$. (If $S=\emptyset$, then the constant map with value $\beta$ will do (it is computed by a $\mathrm{pReLU}$-network of width $2$, for example with first affine map identically zero), so we may also assume that $S\neq\emptyset$.) Let $\epsilon$ be the minimum of $|s-\alpha|$, where $s\in S$. Since $\alpha\not\in S$, we have $\epsilon>0$. Now let $\gamma\in\mathbb Z_p$ be such that $\gamma\neq\alpha$, and 
$$|\gamma-\alpha|<|\beta-\alpha|\cdot\epsilon.$$
Since $|\beta-\alpha|\leq1$, we have $|\gamma-\alpha|<\epsilon$, so
for every $s\in S$ we have $|s-\alpha|\geq\epsilon>|\alpha-\gamma|$, therefore $|s-\gamma|=|(s-\alpha)+(\alpha-\gamma)|=|s-\alpha|\geq\epsilon$ by the ultrametric inequality, and hence
$$\left|\frac{\beta-\alpha}{\alpha-\gamma}(s-\gamma)\right|\geq
\frac{|\beta-\alpha|}{|\gamma-\alpha|}\cdot\epsilon>1.$$
Let $g:\mathbb Z_p\to\mathbb Q_p^2$ be the function given by the rule $g(x)=(\frac{\beta-\alpha}{\alpha-\gamma}(x-\gamma),x)$, let $h:\mathbb Q^2_p\to\mathbb Q_p$ be the function given by the rule $h(x,y)=x+y$, and let $f:\mathbb Z_p\to\mathbb Q_p$ be the composition $h\circ\Sigma\circ g$. By definition $f$ is a $\mathrm{pReLU}$-network of width $2$. If $x=\alpha$, then the first coordinate of $g(x)$ is $\frac{\beta-\alpha}{\alpha-\gamma}(\alpha-\gamma)=\beta-\alpha\in\mathbb Z_p$, so $\Sigma(g(x))=(\beta-\alpha,\alpha)$, and hence $f(x)=\beta-\alpha+\alpha=\beta$. If $x\in S$, then by the above $\Sigma(g(x))=(0,x)$, and hence $f(x)=0+x=x$.
\end{proof}
\begin{lemma}\label{lemma5} Let $S\subset\mathbb Z_p$ be a finite subset, and let $f^*:S\to\mathbb Z_p$ be a function such that $\mathrm{Im}(f^*)\cap S=\emptyset$. Then there is a $\mathrm{pReLU}$-network of width $2$ which computes a function $f:\mathbb Z_p\to\mathbb Z_p$ such that $f(x)=f^*(x)$, if $x\in S$. 
\end{lemma}
\begin{proof} List the elements of $S$ in some order as $\alpha_1,\alpha_2,\ldots,\alpha_n$, and set $\beta_i=f^*(\alpha_i)$ for every $i=1,2,\ldots,n$. For every such $i$ let
$$S_i=\{\beta_j\mid j<i\}\cup\{\alpha_j\mid j>i\}.$$

Because $\mathrm{Im}(f^*)\cap S=\emptyset$ we get that $\alpha_i\not
\in S_i$. Therefore there is a $\mathrm{pReLU}$-network of width $2$ which computes a function $f_i:\mathbb Z_p\to\mathbb Z_p$ such that $f_i(\alpha_i)=\beta_i$ and  $f_i(x)=x$, if $x\in S_i$ by Lemma \ref{lemma4}. Let $g_i$ be the function
$$g_i=f_i\circ f_{i-1}\circ\cdots\circ f_1.$$
Since $g_i$ is the composition of pReLU networks of width $2$, it is a pReLU network of width $2$, too. It will be sufficient to show by induction on $i$ that 
$$g_i(\alpha_j)=\begin{cases}
    \beta_j & \textrm{, if $j=1,2,\ldots,i$},\\
     \alpha_j  & \textrm{, otherwise,} \end{cases}$$
because then $f=g_n$ will satisfy the properties in the claim. We have $g_1=f_1$, so in the case $i=1$ the claim is obvious. Assume now that the statement holds for $i-1$. Note that for every $j=1,2,\ldots,i-1$ the function $g_{i-1}$ maps $\alpha_j$ to $\beta_j$, while $f_i$ maps $\beta_j\in S_i$ to $\beta_j$, therefore $g_i=f_i\circ g_{i-1}$ maps $\alpha_j$ to $\beta_j$. Similarly for every $j=i+1,\ldots,n$ the function $g_{i-1}$ maps $\alpha_j$ to $\alpha_j$, while $f_i$ maps $\alpha_j\in S_i$ to $\alpha_j$, therefore $g_i=f_i\circ g_{i-1}$ maps $\alpha_j$ to $\alpha_j$. Finally $g_{i-1}$ maps $\alpha_i$ to $\alpha_i$ by the induction hypothesis, and $f_i$ maps $\alpha_i$ to $\beta_i$, so $g_i=f_i\circ g_{i-1}$ maps $\alpha_i$ to $\beta_i$.
\end{proof}
\begin{lemma}\label{lemma6} Let $S\subset\mathbb Z_p$ be a finite subset, and let $f^*:S\to\mathbb Z_p$ be a function. Then there is a $\mathrm{pReLU}$-network of width $2$ which computes a function $f:\mathbb Z_p\to\mathbb Z_p$ such that $f(x)=f^*(x)$, if $x\in S$. 
\end{lemma}
\begin{proof} Because $\mathbb Z_p$ is infinite, there is a finite subset $R\subset\mathbb Z_p$ of the same cardinality as $S$ such that $S\cap R=\emptyset$ and $R\cap\mathrm{Im}(f^*)=\emptyset$. Let $g^*:S\to R$ be a bijection and let $h^*:R\to\mathrm{Im}(f^*)$ be the composition $f^*
\circ(g^*)^{-1}$. By Lemma \ref{lemma5} there is are two $\mathrm{pReLU}$-networks $g,h$ of width $2$ such that $g(x)=g^*(x)$, if $x\in S$, and $h(x)=h^*(x)$, if $x\in R$, because $S\cap R=\emptyset$ and $R\cap\mathrm{Im}(f^*)=\emptyset$, respectively. Then the composition $f=h\circ g$  is a $\mathrm{pReLU}$-network of width $2$, since it is the composition of two $\mathrm{pReLU}$-networks of width $2$,  such that $f(x)=f^*(x)$, if $x\in S$.  
\end{proof}
\begin{lemma}\label{lemma2} Let $m$ be a positive integer and let $\alpha\in\mathbb Z_p$. Then there is a $\mathrm{pReLU}$-network of width $2$ which computes the function $f_{\alpha}:\mathbb Z_p\to\mathbb Z_p$ given by
$$f_{\alpha}(x)=\begin{cases}
    \alpha & \text{, if } x\in\alpha+p^m\mathbb Z_p,\\
    x & \text{, otherwise. } \end{cases}$$
\end{lemma}
\begin{proof}  Let $g:\mathbb Z_p\to\mathbb Q_p^2$ be the function given by the rule $g(x)=(\frac{x-\alpha}{p^m},x-\alpha)$, let $h:\mathbb Q^2_p\to\mathbb Q_p$ be the function given by the rule $h(x,y)=y-p^mx+\alpha$, and let $f:\mathbb Z_p\to\mathbb Q_p$ be the composition $h\circ\Sigma\circ g$. By definition $f$ is a $\mathrm{pReLU}$-network of width $2$. If $x\in\alpha+p^m\mathbb Z_p$, then $\frac{x-\alpha}{p^m}\in\mathbb Z_p$, so $\Sigma(g(x))=(\frac{x-\alpha}{p^m},x-\alpha)$, and hence $f(x)=(x-\alpha)-p^m\frac{x-\alpha}{p^m}+\alpha=\alpha$. If $x\not\in\alpha+p^m\mathbb Z_p$, then $\frac{x-\alpha}{p^m}\not\in\mathbb Z_p$, so $\Sigma(g(x))=(0,x-\alpha)$, and hence $f(x)=x-\alpha-0+\alpha=x$.
\end{proof}
\begin{lemma}\label{lemma3} Let $m$ be a positive integer and assume that for every coset $D$ of $p^m\mathbb Z_p$ in $\mathbb Z_p$ an element $\alpha_D\in D$ is chosen. Then there is a $\mathrm{pReLU}$-network of width $2$ which computes the function $f:\mathbb Z_p\to\mathbb Z_p$ such that $f(x)=\alpha_D$, if $x\in\alpha_D+p^m\mathbb Z_p$. 
\end{lemma}
\begin{proof} List the cosets of $p^m\mathbb Z_p$ in $\mathbb Z_p$ in some order as $D(1),D(2),\ldots,D(p^m)$, and set $\alpha_i=\alpha_{D(i)}$ for every $i=1,2,\ldots,p^m$. For every such $i$ let $g_i:\mathbb Z_p\to
\mathbb Z_p$ be the function
$$g_i=f_{\alpha_{i}}\circ f_{\alpha_{i-1}}\circ\cdots\circ f_{\alpha_{1}},$$
where we use the notation in Lemma \ref{lemma2}. Since $g_i$ is the composition of pReLU networks of width $2$, it is a pReLU network of width $2$, too. It will be sufficient to show by induction on $i$ that 
$$g_i(x)=\begin{cases}
    \alpha_j & \textrm{, if $x\in\alpha_j+p^m\mathbb Z_p$ for some $j=1,2,\ldots,i$},\\
    x & \textrm{, otherwise,} \end{cases}$$
because then $f=g_{p^m}$ will satisfy the properties in the claim. We have $g_1=f_{\alpha_1}$, so in the case $i=1$ the claim is obvious. Assume now that the statement holds for $i-1$. Note that for every $j=1,2,\ldots,p^m$ the function $g_{i-1}$ maps $D(j)$ into $D(j)$. For $j\neq i$ the function $f_{\alpha_i}$ is the identity on $D(j)$, so $g_i=
f_{\alpha_i}\circ g_{i-1}$ has the required values on the set $\bigcup_{j\neq i}D(j)$. However $f_{\alpha_i}$ maps $D(i)$ to $\alpha_i$, so  $g_i=f_{\alpha_i}\circ g_{i-1}$ takes the constant value $\alpha_i$ on $D(i)$. 
\end{proof}
\begin{defn} Let $d,m$ be positive integers. An encoding function of type $(d,m)$ is a function $f:\mathbb Z_p^d\to\mathbb Z_p$ such that for every $x,y\in\mathbb Z_p^d$ we have $f(x)=f(y)$ if and only if $x-y\in p^m\mathbb Z_p^d$.
\end{defn}
\begin{lemma}\label{lemma8} Let $m$ be a positive integer and assume that for every coset $D$ of $p^m\mathbb Z_p$ in $\mathbb Z_p$ an element $\alpha_D\in D$ is chosen. Let $f:\mathbb Z_p\to\mathbb Z_p$ be the function such that $f(x)=\alpha_D$, if $x\in\alpha_D+p^m\mathbb Z_p$. Then the function $g:\mathbb Z^d_p\to\mathbb Z_p$ given by the rule:
$$g(x_1,x_2,\ldots,x_d)=\sum_{i=1}^dp^{(i-1)m}f(x_i)=f(x_1)+p^mf(x_2)+\cdots+p^{(d-1)m}f(x_d)$$
for every $(x_1,x_2,\ldots,x_d)\in\mathbb Z_p^d$ is an encoding function of type $(d,m)$. 
\end{lemma}
\begin{proof} Let $(x_1,\dots,x_d)$ and $ (y_1,\dots,y_d)\in\mathbb Z_p^d$ be arbitrary. If $x_i\equiv y_i\mod p^m$ for every $i=1,2,\ldots,d$, then $f(x_i)=f(y_i)$, since $f$ is constant on the cosets of $p^m\mathbb Z_p$, for every such $i$, and hence $g(x_1,\ldots,x_d)=g(y_1,\ldots,y_d)$. Now we only need to show the converse, so let's assume that $g(x_1,\ldots,x_d)=g(y_1,\ldots,y_d)$. It will be sufficient to show that if $x_j\equiv y_j\mod p^m$ for every $j<i$, then $x_i\equiv y_i\mod p^m$, since the claim follows by induction on $i$, as the condition for $i=1$ is empty. 

Since $x_j\equiv y_j\mod p^m$ for every $j<i$, we have $f(x_j)=f(y_j)$ for every such $j$, and hence
$$f(x_1)+p^mf(x_2)+\cdots+p^{m(i-2)}f(x_{i-1})=
f(y_1)+p^mf(y_2)+\cdots+p^{m(i-2)}f(y_{i-1}).$$
Subtracting this from the equation $g(x_1,\ldots,x_d)=g(y_1,\ldots,y_d)$ we get that
$$p^{m(i-1)}f(x_i)+\cdots+p^{m(d-1)}f(x_d)=
p^{m(i-1)}f(y_i)+\cdots+p^{m(d-1)}f(y_d).$$
By looking at this equation mod $p^{mi}$ we get that
\begin{align*}
p^{m(i-1)}f(x_i)&\equiv p^{m(i-1)}f(x_i)+\cdots+p^{m(d-1)}f(x_d)  \\
&\equiv p^{m(i-1)}f(y_i)+\cdots+p^{m(d-1)}f(y_d)\equiv p^{m(i-1)}f(y_i)\mod p^{mi}.
\end{align*}
By dividing the congruence by $p^{m(i-1)}$ we conclude that $x_i\equiv y_i\mod p^m$. 
\end{proof}
\begin{lemma}\label{lemma9} For every pair $d,m$ of positive integers there is a $\mathrm{pReLU}$-network of width $d+1$ which computes an encoding function of type $(d,m)$ on $\mathbb Z_p^d$. 
\end{lemma}
\begin{proof}  Choose an element $\alpha_D\in D$ for every coset $D$ of $p^m\mathbb Z_p$ in $\mathbb Z_p$. By Lemma \ref{lemma3} there is a $\mathrm{pReLU}$-network of width $2$ which computes the function $f:\mathbb Z_p\to\mathbb Z_p$ such that $f(x)=\alpha_D$, if $x\in\alpha_D+p^m\mathbb Z_p$. Therefore it will be sufficient to show that the function $g:\mathbb Z^d_p\to\mathbb Z_p$ given by the rule:
$$g(x_1,x_2,\ldots,x_d)=f(x_1)+p^mf(x_2)+\cdots+p^{(d-1)m}f(x_d)$$
for every $(x_1,x_2,\ldots,x_d)\in\mathbb Z_p^d$ for this choice of $f$ is a 
$\mathrm{pReLU}$-network of width $d+1$ by Lemma \ref{lemma8}. First note that for every $i=1,2,\ldots,d$ the function $g_i:\mathbb Z^d_p\to\mathbb Z^d_p$ given by the rule:
$$g_i(x_1,x_2,\ldots,x_d)=(x_1,\ldots,x_{i-1},f(x_i),x_{i+1},\ldots,x_d)$$
for every $(x_1,x_2,\ldots,x_d)\in\mathbb Z_p^d$ is a $\mathrm{pReLU}$-network of width $d+1$. Then $g$ is the composition of the function
$g_d\circ g_{d-1}\circ\cdots\circ g_1$ and the linear map $h:\mathbb Q_p^d\to\mathbb Q_p$ given by the rule:
$$h(x_1,x_2,\ldots,x_d)=x_1+p^mx_2+\cdots+p^{(d-1)m}x_d,$$
so it is a $\mathrm{pReLU}$-network of width $d+1$, too. 
\end{proof}
\begin{proof}[Proof of Theorem \ref{upper2b}] For simplicity of notation set $d=d_x$ and let $f:\mathbb Z^d_p\to\mathbb Q_p$ be a locally constant function. By part $(ii)$ of Lemma \ref{locally_constant} there is a positive integer $m$ such that $f$ is constant on the cosets of $p^m\mathbb Z^d_p$ in $\mathbb Z^d_p$. Moreover $f$ is bounded, since it is continuous and $\mathbb Z^d_p$ is compact, so there is a non-zero constant $c\in\mathbb Q_p$ such that $f(x)/c\in\mathbb Z_p$ for every $x\in\mathbb Z^d_p$. By Lemma \ref{lemma9} there is a $\mathrm{pReLU}$-network $g:\mathbb Z^d_p\to\mathbb Z_p$ of width $d+1$ which is an encoding function of type $(d,m)$. Let $S\subset\mathbb Z_p$ be the image of $g$; it is a finite subset.

Let $f^*:S\to\mathbb Z_p$ be the function such that for every $s\in S$ we have $f^*(s)=f(x)/c$ for every $x\in g^{-1}(s)$. Since $f/c$ is constant on the cosets of $p^m\mathbb Z^d_p$ in $\mathbb Z^d_p$, the function $f^*$ is well-defined. By Lemma \ref{lemma6} there is a $\mathrm{pReLU}$-network of width $2\leq d+1$ which computes a function $h:\mathbb Z_p\to\mathbb Z_p$ such that $h(x)=f^*(x)$, if $x\in S$. The composition of $g$ and $h$ is a $\mathrm{pReLU}$-network of width $d+1$ which computes $f/c$. By rescaling the last linear function by $c$ we get a $\mathrm{pReLU}$-network of width $d+1$ which computes $f$.
\end{proof}
\begin{lemma}\label{affine} Let $m$ be a positive integer and let
$\alpha\in\mathbb Z_p$. Let $C\subset\mathbb Z_p$ be the convex subset $\alpha+p^m\mathbb Z_p$, and let $f:C\to\mathbb Q_p$ be an affine map. Then $f$ maps $C$ surjectively onto $\mathbb Z_p$ if and only if $f$ is of the form $f(x)=a(x-\alpha)+b$, where $a\in\mathbb Q_p$ is such that $|a|=p^m$ and $b\in\mathbb Z_p$. 
\end{lemma}
\begin{proof} First assume that $f$ is of the form $f(x)=a(x-\alpha)+b$, where $|a|=p^m$ and $b\in\mathbb Z_p$. For every $x\in C$ we have
$|x-\alpha|\leq\frac{1}{p^m}$, and hence
$$|a(x-\alpha)+b|\leq
\max(|a(x-\alpha)|,|b|)\leq\max(p^m\cdot\frac{1}{p^m},1)=1.$$
Therefore $f$ maps $C$ into $\mathbb Z_p$. Since the inverse of $f$ is $f^{-1}(y)=a^{-1}(y-b)+\alpha$, we need to show that the latter is in $C$ when $y\in\mathbb Z_p$ is arbitrary. But the latter holds:
$$|(a^{-1}(y-b)+\alpha)-\alpha|=|a^{-1}|\cdot|y-b|\leq
\frac{1}{p^m}\cdot\max(|y|,|b|)\leq\frac{1}{p^m}\cdot1=\frac{1}{p^m}.$$
Now assume that $f$ maps $C$ surjectively onto $\mathbb Z_p$. Since $f$ is affine, it can be written in the form $f(x)=a(x-\alpha)+b$ where $a,b\in\mathbb Q_p$. Because $f(\alpha)=b$, we get that $b\in\mathbb Z_p$. Since $x-\alpha$ ranges over $p^m\mathbb Z_p$ as $x$ ranges over $C$, we have $f(C)=b+ap^m\mathbb Z_p$. If this is equal to $\mathbb Z_p$, then $a\neq0$, since otherwise $f(C)$ is a single point, and since $b\in\mathbb Z_p$ we get that $ap^m\mathbb Z_p=\mathbb Z_p$. Therefore $|a|\cdot p^{-m}=|ap^m|=1$, that is, $|a|=p^m$.
\end{proof}
\begin{lemma}\label{lemma9a} Let $m$ be a positive integer and let
$D\subset\mathbb Z_p$ be a coset of $p^m\mathbb Z_p$. Let $C\subset D$ be convex and let $f:C\to\mathbb Z_p$ be a surjective affine map. If $\beta\in\mathbb Z_p$ then there is a convex subset $B\subset C$ such that the affine function $g:B\to\mathbb Q_p$ given by the rule $g(x)=\frac{f(x)}{p^m}+\frac{x-\beta}{p^m}$ maps $B$ surjectively onto
$\mathbb Z_p$.
\end{lemma}
\begin{proof} Since $\beta\in\mathbb Z_p$, the subset $D'=\beta-D=\{\beta-c\mid c\in D\}$ is a coset of $p^m\mathbb Z_p$ in $\mathbb Z_p$, too.  Because $f$ surjects $C$ onto $\mathbb Z_p$ there is an $\alpha\in C$ such that $f(\alpha)\in D'$. The set $C-\alpha$ is a convex subset of $\mathbb Z_p$ containing $0$, hence a $\mathbb Z_p$-submodule of $\mathbb Z_p$, that is, an ideal. It is contained in $D-\alpha=p^m\mathbb Z_p$, and it is non-zero, since $f$ maps $C$ surjectively onto $\mathbb Z_p$. Therefore $C=\alpha+p^n\mathbb Z_p$ for some integer $n\geq m$. By Lemma \ref{affine} we have $f(x)=a(x-\alpha)+b$, where $a\in\mathbb Q_p$ is such that $|a|=p^n$. Then we can write $g$ in the form
$$g(x)=\frac{a+1}{p^m}(x-\alpha)+\frac{\alpha-\beta+b}{p^m}.$$
Since $b=f(\alpha)\in D'$ and $\alpha\in D$, and hence 
$$D'=\beta-\alpha+p^m\mathbb Z_p,$$
there is a $d\in p^m\mathbb Z_p$ such that $b=\beta-\alpha+d$. So we get that
$$|\frac{\alpha-\beta+b}{p^m}|=|\frac{d}{p^m}|\leq1,\textrm{ and hence }
\frac{\alpha-\beta+b}{p^m}\in\mathbb Z_p.$$
Since $|a|=p^n>1$, by the ultrametric identity $|a+1|=p^n$, and hence $|\frac{a+1}{p^m}|=p^{m+n}$. Therefore $g$ maps $\alpha+p^{m+n}\mathbb Z_p\subset C$ surjectively onto $\mathbb Z_p$ by Lemma \ref{affine}. 
\end{proof}
\begin{defn} Let $m$ be a positive integer. A juggling function of type $m$ is a function $g:\mathbb Z_p\to\mathbb Z_p$ such that for every $y\in\mathbb Z_p$ and for every coset $D$ of $p^m\mathbb Z_p$ in
$\mathbb Z_p$ we have $g^{-1}(y)\cap D\neq\emptyset$.
\end{defn}
\begin{lemma}\label{lemma9b} For every positive integer $m$ there is a  juggling function $g:\mathbb Z_p\to\mathbb Z_p$ of type $m$ such that the function $f:\mathbb Z_p\to\mathbb Z_p^2$ given by the rule $f(x)=(x,g(x))$ is computed by a $\mathrm{pReLU}$-network of width $2$.
\end{lemma}
\begin{proof} For every $\beta\in\mathbb Z_p$ Let $h_{\beta}:\mathbb Z^2_p\to\mathbb Q_p^2$ be the function given by the rule $h_{\beta}(x,y)=(x,\frac{x-\beta}{p^m}+\frac{y}{p^m})$, and let $t_{\beta}:\mathbb Z^2_p\to\mathbb Z^2_p$ be the composition $\Sigma\circ h_{\beta}$. By definition $t_{\alpha}$ is a $\mathrm{pReLU}$-network of width $2$. Let $\beta_1,\beta_2,\ldots,\beta_{p^m}$ be a list of coset representatives of $p^m\mathbb Z_p$ in $\mathbb Z_p$. Let $f_0:\mathbb Z_p\to\mathbb Z_p^2$ be the function given by the rule $f_0(x)=(x,0)$, and for every $i=1,2,\ldots,p^m$ let $f_i:\mathbb Z_p\to\mathbb Z_p^2$ be the composition of $f_0$ and the function 
$$t_{\beta_i}\circ t_{\beta_{i-1}}\circ\cdots\circ t_{\beta_1}.$$
Clearly $f_i$ is a $\mathrm{pReLU}$-network of width $2$ such that $f(x)=(x,g_i(x))$ for some function $g_i:\mathbb Z_p\to\mathbb Z_p$. For $i=0,1,2,\ldots,p^m$ we are going to show by induction on $i$ that
\begin{enumerate}
\item[$(a)$] for every $j\leq i$ there is a convex subset $B_{i,j}\subset
\beta_j+p^m\mathbb Z_p$ such that restriction of $g_i$ onto $B_{i,j}$ is an affine function which maps $B_{i,j}$ surjectively onto $\mathbb Z_p$,
\item[$(b)$] for every $j> i$ the restriction of $g_i$ onto
$\beta_j+p^m\mathbb Z_p$ is identically zero. 
\end{enumerate}
This claim implies that $f_{p^m}$ is a function of the type in the claim of the lemma, so it will be sufficient to show it. The claim is obvious for $i=0$, so assume now that the statement holds for $i-1$. If $j>i$ then $g_{i-1}$ is identically zero on $\beta_j+p^m\mathbb Z_p$, hence $h_{\beta_i}(x,g_{i-1}(x))=(x,\frac{x-\beta_i}{p^m})$ on the latter set. For any $x\in \beta_j+p^m\mathbb Z_p$ we have $\frac{x-\beta_i}{p^m}\not\in\mathbb Z_p$, so $t_i(x)=\Sigma\circ h_{\beta}(x,g_{i-1}(x))=(x,0)$ on $\beta_j+p^m\mathbb Z_p$.

Similarly $g_{i-1}$ is identically zero on $\beta_i+p^m\mathbb Z_p$, but for any $x\in \beta_i+p^m\mathbb Z_p$ we have $\frac{x-\beta_i}{p^m}\in\mathbb Z_p$, so $t_i(x)=\Sigma\circ h_{\beta}(x,g_{i-1}(x))=(x,\frac{x-\beta_i}{p^m})$ on $\beta_i+p^m\mathbb Z_p$. Therefore the restriction of $g_i$ onto $B_{i,i}=\beta_i+p^m\mathbb Z_p$ is an affine function which maps $B_{i,i}$ surjectively onto $\mathbb Z_p$ by Lemma \ref{affine}. Finally for $j<i$ the restriction of $g_{i-1}$ onto $B_{i-1,j}\subset\beta_j+p^m\mathbb Z_p$ is an affine function which maps $B_{i-1,j}$ surjectively onto $\mathbb Z_p$, so by Lemma \ref{lemma9a} the function $\frac{g_{i-1}(x)}{p^m}+\frac{x-\beta_i}{p^m}$ maps a convex subset $B_{i,j}\subset B_{i-1,j}$ surjectively onto
$\mathbb Z_p$ and is affine on this set. Because $g_i=\frac{g_{i-1}(x)}{p^m}+\frac{x-\beta_i}{p^m}$ on this set, as the latter takes values in $\mathbb Z_p$, the proof is now complete.  
\end{proof}
\begin{defn} Let $d,m$ be positive integers. A decoding function of type $(d,m)$ is a function $f:\mathbb Z_p\to\mathbb Z_p^d$ such that for every $y\in\mathbb Z_p^d$ there is an $x\in\mathbb Z_p$ such that $y-f(x)\in p^m\mathbb Z^d_p$. 
\end{defn}
\begin{lemma}\label{lemma9c} Let $g:\mathbb Z_p\to\mathbb Z_p$ be a
juggling function of type $m$ for a positive integer $m$. Then $f:\mathbb Z_p\to\mathbb Z_p^d$ given by the rule 
$$f(x)=(x,g(x),g\circ g(x),\ldots,
\underbrace{g\circ\cdots\circ g}_{\textrm{$d-1$ times}}(x))$$
is a decoding function of type $(d,m)$ on $\mathbb Z_p$.
\end{lemma}
\begin{proof} Let $(y_1,y_2,\ldots,y_d)\in\mathbb Z_p^d$ be arbitrary. We are going to construct a sequence of points $x_i\equiv y_i\mod p^m$ for every $i=1,\ldots,d$ and $x_{i+1}=g(x_i)$ for every $i=1,\ldots,d-1$ by descending recursion on the index $i$. This implies that
$$f(x_1)=(x_1,g(x_1),g\circ g(x_1)\ldots,\underbrace{g\circ\cdots\circ g}_{\textrm{$d-1$ times}}(x_1))=(x_1,x_2,\ldots,x_d)$$
is in the same $p^m\mathbb Z_p^d$-coset as $(y_1,y_2,\ldots,y_d)$. Because the latter is arbitrary, we get that the existence of the sequence $x_d,x_{d-1},\ldots$ above implies the claim. First set $x_d=y_d$. If $x_d,x_{d-1},\ldots,x_{i+1}$ has already been constructed, then choose $x_i$ to be an element of $g^{-1}(x_{i+1})\cap(y_i+p^m\mathbb Z_p)$. The latter is possible because $g$ is a juggling function of type $m$. Clearly $x_i$ satisfies the required conditions $x_i\equiv y_i\mod p^m$ and $g(x_i)=x_{i+1}$. 
\end{proof}
\begin{lemma}\label{lemma10} For every pair $d,m$ of positive integers there is a $\mathrm{pReLU}$-network of width $d$ which computes a decoding function of type $(d,m)$ on $\mathbb Z_p$. 
\end{lemma}
\begin{proof} By Lemma \ref{lemma9b} there is a juggling function $g:\mathbb Z_p\to\mathbb Z_p$ of type $m$ such that the function $f:\mathbb Z_p\to\mathbb Z_p^2$ given by the rule $f(x)=(x,g(x))$ is computed by a $\mathrm{pReLU}$-network of width $2$. For every positive integer $d$ let $f_d:\mathbb Z_p\to\mathbb Z_p^d$ be the function given by the rule:
$$f_d(x)=(x,g(x),g\circ g(x),\ldots,
\underbrace{g\circ\cdots\circ g}_{\textrm{$d-1$ times}}(x)).$$
By Lemma \ref{lemma9c} it will be sufficient to prove that $f_d$ is computed by a $\mathrm{pReLU}$-network of width $d$ by induction on $d$. As $f_1$ is the identity function on $\mathbb Z_p$, the case $i=1$ is clear. For $d\geq2$ the function $f_d$ is the composition of $f_{d-1}$ and 
the function $g_d:\mathbb Z^{d-1}_p\to\mathbb Z^d_p$ given by the rule:
$$g_d(x_1,x_2,\ldots,x_{d-1})=(x_1,\ldots,x_{d-1},g(x_{d-1}))$$
for every $(x_1,x_2,\ldots,x_d)\in\mathbb Z_p^{d-1}$. The latter is a
$\mathrm{pReLU}$-network of width $d$ by Lemma \ref{lemma9b}, so $f_d$ is the composition of a $\mathrm{pReLU}$-network of width $d-1$ and a $\mathrm{pReLU}$-network of width $d$, so it is a $\mathrm{pReLU}$-network of width $d$. 
\end{proof}
We are finally ready to give a full proof of Theorem \ref{upper}. Let $f:\mathbb Z^{d_x}_p\to\mathbb Q_p^{d_y}$ be a continuous function which we want to approximate arbitrarily well with $\mathrm{pReLU}$-network of width $\max(d_x+1,d_y)$ in any $L_q$ norm for each $q\in[1,\infty]$. As we already remarked we will only need to do this for the $L_{\infty}=C_1$ norm. By part $(i)$ of Lemma \ref{locally_constant} we may assume without the loss of generality that $f$ is a locally constant function. By part $(ii)$ of Lemma \ref{locally_constant} there is a positive integer $m$ such that $f$ is constant on the cosets of $p^m\mathbb Z_p^{d_x}$ in $\mathbb Z_p^{d_x}$.

Since $f$ is continuous and $\mathbb Z_p^{d_x}$ is compact, $f$ is bounded, so there is a non-zero constant $c\in\mathbb Q_p$ with $|c|\geq\sup_x\|f(x)\|$ such that $f(x)/c\in\mathbb Z^{d_y}_p$ for every $x\in\mathbb Z^{d_x}_p$. If we can approximate $f/c$ 
arbitrarily well with $\mathrm{pReLU}$-networks of width
$\max(d_x+1,d_y)$ in the $C_1$ norm then the $\mathrm{pReLU}$-networks of width $\max(d_x+1,d_y)$ which we get by rescaling the last linear function by $c$ will approximate $f$ arbitrarily well in the $C_1$ norm. Therefore me may assume without the loss of generality that $f$ takes values in $\mathbb Z_p^{d_y}$. 

For every positive integer $n$ we will construct a $\mathrm{pReLU}$-network $f^*:\mathbb Z_p^{d_x}\to\mathbb Z_p^{d_y}$ of width
$\max(d_x+1,d_y)$ such that $\|f-f^*\|_{\infty}\leq\frac{1}{p^n}$ as follows. By Lemma \ref{lemma10} there is a $\mathrm{pReLU}$-network $g:\mathbb Z_p\to\mathbb Z^{d_y}_p$ of width $d_y$ which is a decoding function of type $(d_y,n)$. Let $S\subset\mathbb Z_p$ be a finite subset such that for every coset $C$ of $p^n\mathbb Z_p^{d_y}$ in $\mathbb Z_p^{d_y}$ there is exactly one $s\in S$ such that $g(s)\in C$. For every coset $D$ of $p^m\mathbb Z_p^{d_x}$ in $\mathbb Z_p^{d_x}$ let $h^*(D)\in S$ be the unique element such that $f$ maps $D$ into the $p^n\mathbb Z_p^{d_y}$-coset of $g(h^*(D))$ in $\mathbb Z_p^{d_y}$. 

By Theorem \ref{upper2b} there is a $\mathrm{pReLU}$-network 
$h:\mathbb Z_p^{d_x}\to\mathbb Z_p$ of width $d_x+1$ which computes a function $h:\mathbb Z^{d_x}_p\to\mathbb Z_p$ such that for every coset $D$ of $p^m\mathbb Z_p^{d_x}$ in $\mathbb Z_p^{d_x}$ we have $h(x)=h^*(D)$, if $x\in D$. The composition of $h$ and $g$ is a $\mathrm{pReLU}$-network $f^*:\mathbb Z_p^{d_x}\to\mathbb Z_p^{d_y}$ of width $\max(d_x+1,d_y)$ which is constant on the cosets of $p^m\mathbb Z_p^{d_x}$ in $\mathbb Z_p^{d_x}$ and for every such coset $D$ the value of $f$ and $f^*$ on $D$ lies in the same coset of $p^n\mathbb Z_p^{d_y}$ in $\mathbb Z_p^{d_y}$ by construction. Therefore $\|f-f^*\|_{\infty}\leq\frac{1}{p^n}$.
\end{proof}

\end{document}